\documentclass[reqno]{amsart}
\usepackage{amsmath,amssymb,amsfonts,caption}
\usepackage{graphicx}
\usepackage{euscript}
\usepackage{enumerate}
\usepackage{verbatim}
\usepackage{epsfig}

\newtheorem{theorem}{Theorem}
\newtheorem{lemma}{Lemma}
\newtheorem{proposition}{Proposition}

\renewcommand{\epsilon}{\varepsilon}
\renewcommand{\phi}{\varphi}
\renewcommand{\le}{\leqslant}
\renewcommand{\ge}{\geqslant}
\newcommand{\eps}{\varepsilon}
\newcommand{\ds}{\, ds}

\newcommand{\du}{\, du}
\newcommand{\dtau}{\, d\tau}

\newcommand{\dint}{\displaystyle \int}
\newcommand{\per}{PER}
\newcommand{\aut}{AUT}

\DeclareMathOperator{\e}{e}

\usepackage{dsfont}
   \newcommand{\N}{\ensuremath{\mathds N}}
   \newcommand{\R}{\ensuremath{\mathds R}}

\begin{document}
\title[]
   {A nonautonomous epidemic model with general incidence and isolation}
\author{C\'esar M. Silva}
\address{C. Silva\\
   Departamento de Matem\'atica\\
   Universidade da Beira Interior\\
   6201-001 Covilh\~a\\
   Portugal}
\email{csilva@ubi.pt}
\date{\today}
\thanks{Work partially supported by FCT thought CMUBI (project PEst-OE/MAT/UI0212/2011)}
\begin{abstract}
We obtain conditions for eradication and permanence of infection for a nonautonomous SIQR model with time-dependent parameters, that are not assumed to be periodic. The incidence is given by functions of all compartments and the threshold conditions are given by some numbers that play the role of the basic reproduction number. We obtain simple threshold conditions in the autonomous, asymptotically autonomous and periodic settings and show that our thresholds coincide with the ones already established. Additionally, we obtain threshold conditions for the general nonautonomous model with mass-action, standard and quarantine-adjusted incidence.
\end{abstract}
\maketitle
\section{Introduction}
Isolation and quarantine are two possible measures for reducing the transmission of diseases and have been used over the centuries to reduce the transmission of human diseases such as leprosy, plague, typhus, cholera, yellow fever and tuberculosis.

The concept of quarantine dates back to the fourteenth century and is related to plague. In fact, in 1377, the Rector of the seaport of the old town of Ragusa (modern Dubrovnik) officially issued a thirty-day isolation period: ships coming from infected or suspected to be infected sites were to stay at anchor for thirty days before docking. A similar measure was adopted for land travelers but the period of time was enlarged to forty days. The word quarantine comes from the Italian word ``quarantena'', meaning forty-day period.

In the context of epidemic models, the term quarantine is used to describe the deliberate separation of individuals suspected of being exposed to a disease, from a population of susceptible individuals. Progression to a symptomatic (infectious) stage results in isolation, that is, in the strict separation of an individual from the members of the population at risk. Isolation is still an effective way of controlling infectious diseases. For some milder infectious diseases, quarantined individuals could be individuals that choose to stay home from school or work while for more severe infectious diseases people may be forced into isolation.

To study the effect of isolation in the spread of a disease, we consider a SIQR model. Thus, we divide the population into four compartments: the $S$ compartment corresponding to uninfected individuals that are susceptible to the disease, the $I$ compartment corresponding to individuals that are infected and not isolated, the $Q$ compartment corresponding to isolated individuals ($Q$ stands for quarantine) and the $R$ compartment including the recovered and immune individuals. We assume in this paper that the isolated individuals are perfectly separated from the others so that they do not infect the susceptibles and that the infection confers permanent immunity upon recovery. Measles, mumps and rubella are examples of diseases that confer lifelong immunity.

A major concern when studying epidemic models is to understand the asymptotic behavior of the compartments considered, particularly the infective compartment. Frequently, in autonomous epidemic models, one of two situations occur: there is a disease-free equilibrium that is the unique equilibrium of the system or the disease-free equilibrium coexists with an endemic equilibrium. The basic reproduction number is a fundamental tool to obtain the behavior for a given set of parameters: if the basic reproductive number is less than one the disease-free equilibrium is locally asymptotically stable and if this number exceeds one the disease remains in the population.

Several authors studied autonomous epidemic models with a quarantine class.
Namely, Feng and Thieme~\cite{Feng-Thieme-MB-1995} proposed a SIQR model for childhood diseases and showed that isolation can be responsible for the existence of self-sustained oscillations. Their model includes a modified incidence function - the quarantine-adjusted incidence - where the number of per capita contacts between susceptibles and infectives is divided by the number of non isolated individuals. Some dynamical aspects of this SIQR model were studied by Wu and Feng~\cite{Wu-Feng-JDE-2000}. Namely, these authors studied this model through unfolding analysis of a normal form derived from the model and found Hopf and homoclinic bifurcations associated to this unfolding. A generalized version of Feng and Thieme's model, allowing infected individuals to recover without passing through quarantine and also disease related deaths, was proposed by Hethcote, Zhien and Shengbing in~\cite{Hethcote-Zhien-Shengbing-MB-2002}, where this model was compared to SIQR models with different incidences. A similar analysis is also undertaken in that paper for SIQS models with different incidences. The authors concluded in that work that the SIQR model with quarantine-adjusted incidence seems consistent with the sustained oscillations that are observed in real disease incidences but do not match every aspect of the observed data.

Other types of epidemic models with quarantine were studied in the literature. A SEIQR model was considered by Gerberry and Milner in~\cite{Gerberry-Milner-JMB-2009} where the authors studied the dynamics and used historical data to discuss the adequacy of the model in the context of childhood diseases. The global dynamic behavior of a different SEIQS epidemic model, with a nonlinear incidence of the form $\beta SI$, was analyzed by Yi, Zhao and Zhang in~\cite{Yi-Zhao-Zhang-IJISS-2009}. In another direction, Arino, Jordan and van den Driessche~\cite{Arino-Jordan-Driessche-MB-2007} studied the effect of quarantine in a general model of a disease that can be transmitted between different species and multiple patches (though in their setting quarantine is present only in the form of travel restriction between patches). A model with twelve different classes was proposed by Safi and Gumel~\cite{Safi-Gumel-CSF-2007} to study the impact of quarantine of
latent cases, isolation of symptomatic cases and of an imperfect vaccine in the control of the spread of an infectious disease in the population. The same authors studied a delayed model including a quarantined class~\cite{Safi-Gumel-NA-2011}.

We emphasize that all the above models are autonomous models. However, it is well-known that fluctuations are very common in disease transmission. Many diseases show seasonal behaviour, with some periods of the year being more propitious for transmission than others. For instance, the opening and closing of schools tend to make the contact rates vary seasonally in childhood diseases~\cite{Dietz-LNB-1976}. For example, the study of weekly measles case reports from England and Wales demonstrated a decline in the transmission parameter during school holidays~\cite{Fine-Clarkson-IJE-1982}.

It has been demonstrated that nonautonomous behavior, in particular periodic or almost periodic behavior, is present in contact rates corresponding to diseases like the flu~\cite{Earn-Dushoff-Levin-TEE-2020}, measles, chickenpox and mumps~\cite{London-Yorke-AJE-1973}, mucormycosis~\cite{Ajam-Bizri-Mokhbat-Weedon-Lutwick-EI-2006}, leishmaniasis~\cite{Bacaer-Guernaoui-JMB-2006} and dengue~\cite{Coutinho-Burattini-Lopez-Massad}. Thus, in the study of seasonal infections, it is natural to consider contact rates given by periodic functions. A common contact rate can have the form $\beta(t)=\alpha(1+\sigma\cos(2\pi t))$, where $\sigma$ is the amplitude of seasonal variation in transmission, usually called ``strength of seasonal forcing''~\cite{Grassly-Fraser-PRSB-2006}.
Additionally, it is also a fact that environment and demographics change with time, sometimes in a aperiodic way. It is therefore natural to consider time-dependent parameters in epidemic models and define threshold conditions for extinction and persistence of the disease in the nonautonomous case.

Our objective in this paper is to consider a family of nonautonomous SIQR models with a large family of contact rates given by general functions that are, in general, time-dependent and far from bilinear, and obtain numbers that play the role of the basic reproductive number in this setting. Related results for nonautonomous SIRVS models were obtained in~\cite{Pereira-Silva-Silva, Zhang-Teng-Gao-AA-2008}, though the incidences considered in those papers are linear in the infectives while here we assume that the incidences are, in general, nonlinear functions of all compartments. We not that, for some diseases with long latent stages, a latent class must be considered. Thus our model is only suitable for diseases with zero or negligible latent period.

Note that in~\cite{Rebelo-Margheri-Bacaer-JMB-2012, Wang-Zhao-JDDE-2008} some methods were developed to obtain thresholds for general periodic models. These methods apply to our setting when we assume the coefficients to be periodic. We emphasise that this is not necessarily our case, since our coefficients need not be periodic. Additionally, our thresholds are given by explicit formulas (see section~\ref{section:SR}).

The structure of the paper is as follows: in section~\ref{section:NP} we introduce the model and the main definitions; in section~\ref{section:SR} we state our main Theorem; in the next sections we look at some particular cases of our model, namely, in section~\ref{section:RNAS} we particularize for the autonomous model, in section~\ref{section:NSIQRSIR} we see that the asymptotically autonomous model has the same thresholds as the limiting autonomous model, in section~\ref{section:PER} we obtain thresholds for the periodic model with constant death and recruitment rates, in section~\ref{section:GEN} we consider the general nonautonomous model with mass-action and constant birth and death rates, in section~\ref{section:GENSQA} we consider the general nonautonomous model with standard and quarantine-adjusted incidences; in section~\ref{section:PT} we prove some auxiliary results and the main theorem; in section~\ref{section:SIM} we present some simulations that illustrate our results; finally in section~\ref{section:D} we summarize our results and consider some possible future work.
\section{Generalized nonautonomous SIQR model} \label{section:NP}
We consider the following nonautonomous and general incidence SIQR model that generalizes the three autonomous SIQR models in~\cite{Hethcote-Zhien-Shengbing-MB-2002},
\begin{equation}\label{eq:ProblemaPrincipal}
\begin{cases}
S'=\Lambda(t)-\phi(t,S,R,Q,I)-d(t) S \\
I'=\phi(t,S,R,Q,I)-\left[\gamma(t)+\sigma(t)+d(t)+\alpha_1(t) \right]I \\
Q'=\sigma(t)I -[d(t)+\alpha_2(t)+\eps(t)]Q \\
R'=\gamma(t) I + \eps(t) Q -d(t) R
\end{cases},
\end{equation}
where $S$, $R$, $Q$ and $I$ correspond respectively to the susceptible, recovered, quarantine and infective compartments; $\Lambda(t)$ is the recruitment of susceptible (births and immigration); $\phi(t,S,R,Q,I)$ is the incidence (into the infective class); $d(t)$ is the per capita natural mortality rate; $\gamma(t)$ is the recovery rate; $\sigma(t)$ is the removal rate from the infectives; $\alpha_1(t)$ is the disease related death in the infectives; $\alpha_2(t)$ is the disease related death in the quarantine class; and $\epsilon(t)$ is the removal rate from the quarantine class. We will assume that $\phi,\Lambda, d, \gamma, \sigma, \alpha_1, \alpha_2$ and $\epsilon$ are continuous bounded and nonnegative functions on $\R_0^+$ and that there are $\omega_d,\omega_\Lambda >0$ such that
   \begin{equation}\label{eq:d-Lambda-}
   d_{\omega_d}^-> 0 \quad \text{and} \quad \Lambda_{\omega_\Lambda}^- > 0,
   \end{equation}
where we are using the notation
\begin{equation}\label{eq:asymptotic-parameters}
  \quad h_{\omega_h}^-=\liminf_{t \to +\infty} \int_t^{t + \omega_h} h(s) \ds \quad \text{and} \quad h_{\omega_h}^+=\limsup_{t \to +\infty} \int_t^{t+\omega_h} h(s) \ds,
\end{equation}
that we will keep on using throughout the paper. For bounded $h$ we will also use the notation
$$h_S=\sup_{t \ge 0} h(t) \quad \text{ and } \quad h_I=\inf_{t \ge 0} (t).$$

Next we state some simple facts about problem~\eqref{eq:ProblemaPrincipal}.
\begin{proposition}\label{subsection:BS}
We have the following:
\begin{enumerate}[i)]
\item \label{cond-1-bs} all solutions $(S(t),I(t),Q(t),R(t))$ of~\eqref{eq:ProblemaPrincipal}
with $S(t_0)\ge 0$, $I(t_0)\ge 0$, $Q(t_0)\ge 0$ and $R(t_0))\ge 0$ verify $S(t)\ge 0$, $I(t)\ge 0$, $Q(t)\ge 0$ and $R(t))\ge 0$ for all $t \ge t_0$;
\item \label{cond-2-bs} all solutions $(S(t),I(t),Q(t),R(t))$ of~\eqref{eq:ProblemaPrincipal}
with $S(t_0)> 0$, $I(t_0)> 0$, $Q(t_0)> 0$ and $R(t_0))> 0$ verify $S(t)> 0$, $I(t)> 0$, $Q(t)> 0$ and $R(t))> 0$ for all $t \ge t_0$;
\item \label{cond-3-bs}  If $(S(t),I(t),Q(t),R(t))$ is a solution
of~\eqref{eq:ProblemaPrincipal} with nonnegative initial conditions then there is a constant $K>0$ such that
\[
\limsup_{t \to +\infty} \, ( S(t)+I(t)+Q(t)+R(t) ) \le K.
\]
\item \label{cond-4-bs} There are constants $C,T>0$ such that, if $I(t)<\beta$ for all $t>T$ then
\begin{equation}\label{eq:boundQR}
Q(t)<C\beta \quad \text{ and } \quad R(t)<C\beta
\end{equation}
for all $t$ sufficiently large.
\end{enumerate}
\end{proposition}

\begin{proof}
Properties~\ref{cond-1-bs}),~\ref{cond-2-bs}) are a simple consequence of the nonnegativeness and the boundedness of the parameter functions. Using~\eqref{eq:d-Lambda-}, property~\ref{cond-3-bs}) follows easily.

Assuming that $I(t)<\beta$ for all $t \ge \widetilde{T}$, by the third equation in~\eqref{eq:ProblemaPrincipal}, we get
   $$Q'(t) < \beta \sigma(t) - [d(t)+\alpha_2(t)+\eps(t)]Q(t)$$
and thus, by~\eqref{eq:d-Lambda-}, for some $T>\widetilde{T}$ sufficiently large and $C_0=\frac{d_{\omega_d}^-}{2 \omega_d}$, we have
   \[
   \begin{split}
   Q(t)
   & = \e^{-\int_{T}^t d(s)+\alpha_2(s)+\eps(s) \ds} Q(T) + \int_{T}^t \e^{-\int_u^t d(s)+\alpha_2(s)+\eps(s) \ds} \beta \sigma(u) \du \\
   & \le \e^{-C_0 (t-T-\omega_d)} Q(T) + \beta \sigma_S \int_{T}^t \e^{-C_0(t-u-\omega_d)} \du \\
   & \le \e^{-C_0 (t-T-\omega_d)} K + \frac{\beta \sigma_S \e^{C_0\omega_d}}{C_0}\left(1-\e^{-C_0(t-T)}\right) \\
   \end{split}
   \]
and thus, taking $C_1=\sigma_S \e^{C_0\omega_d}/(2C_0)$, we have $Q(t)\le C_1 \beta$ for $t$ sufficiently large. A similar argument applied to the third equation in~\eqref{eq:ProblemaPrincipal} shows that there is a constant $C_2>0$ (independent of $\beta$) such that $R(t) \le C_2 \beta$ for $t$ sufficiently large. Choosing $C=\max\{C_1,C_2\}$ we obtain property~\ref{cond-4-bs}).
\end{proof}

Consider the auxiliary equation
\begin{equation}\label{eq:SistemaAuxiliar}
x'=\Lambda(t)-d(t)x,
\end{equation}
where $\Lambda$ and $d$ are the birth and death rate in~\eqref{eq:ProblemaPrincipal} (that we have assumed to be continuous bounded and nonnegative functions on $\R_0^+$ verifying $d_{\omega_d}^-> 0$ and $\Lambda_{\omega_\Lambda}^- > 0$ for some $\omega_d,\omega_\Lambda >0$). Equation~\eqref{eq:SistemaAuxiliar} corresponds to the differential equation obtained by adding the equations in~\eqref{eq:ProblemaPrincipal} and thus solutions of equation~\eqref{eq:SistemaAuxiliar} correspond to the evolution of the total population. The solutions of this equation will be important to determine threshold conditions for the permanence and extinction of the disease. In the next proposition some properties of the solutions of equation~\eqref{eq:SistemaAuxiliar} are established.

\begin{proposition} \label{eq:subsectionAS}
We have the following:
\begin{enumerate}[i)]
\item \label{cond-1-aux}
Given $t_0 \ge 0$, all solutions $x(t)$ of equation~\eqref{eq:SistemaAuxiliar} with initial condition $x(t_0) \ge 0$ are nonnegative for all $t \ge 0$;
\item \label{cond-1a-aux}
Given $t_0 \ge 0$, all solutions $x(t)$ of equation~\eqref{eq:SistemaAuxiliar} with initial condition $x(t_0) > 0$ are positive for all $t \ge 0$;
\item \label{cond-1b-aux}
Given $t_0 \ge 0$ and a solution $x(t)$ of equation~\eqref{eq:SistemaAuxiliar} with initial condition $x(t_0) > 0$, there are $M>0$ and $t_1 \ge 0$ such that $x(t)\ge M$ for all $t \ge t_1$;
\item \label{cond-2-aux} Each fixed solution $x(t)$ of~\eqref{eq:SistemaAuxiliar} with initial condition $x(t_0) \ge 0$ is bounded and globally uniformly attractive on $[0,+\infty)$;
\item \label{cond-3-aux} There is a constant $D>0$ such that, if $x(t)$ is a solution of~\eqref{eq:SistemaAuxiliar} and $\tilde x(t)$ is a solution of the system
\begin{equation}\label{eq:sist-aux-f}
x'=\Lambda(t) - d(t) x + f(t)
\end{equation}
with $\tilde x(t_0)=x(t_0)$ then
\[
\sup_{t \ge t_0} |\tilde x(t) - x(t)| \le D \ \sup_{t \ge t_0} |f(t)|.
\]

\begin{proof}
For $t_0 \ge 0$, the solution of~\eqref{eq:SistemaAuxiliar} with initial condition $x(t_0)=x_0>0$ is given by
\begin{equation}\label{eq:sol-aux}
   x(t)=\e^{-\int_{t_0}^t d(s) \ds} x_0 +\dint_{t_0}^t \e^{-\int_u^t
   d(s)\ds} \Lambda(u) \du.
\end{equation}
Since $t \mapsto \Lambda(t)$ is nonnegative, equation~\eqref{eq:sol-aux} immediately implies~\ref{cond-1-aux}) and~\ref{cond-1a-aux}).

Since $t \mapsto d(t)$ is bounded, by and~\eqref{eq:sol-aux} and~\eqref{eq:d-Lambda-} we obtain
   \[
   \begin{split}
   x(t)
   & = \e^{-\int_{t-\omega_\Lambda}^t d(s) \ds} x(t-\omega_\Lambda) +\dint_{t-\omega_\Lambda}^t
   \e^{-\int_u^t d(s)\ds} \Lambda(u) \du\\
   &>\e^{-d_S\omega_\Lambda}\dint_{t-\omega_\Lambda}^t \Lambda(u)\du\\
   & \ge \e^{-d_S \omega_\Lambda} \Lambda_{\omega_\Lambda}^-/2,
   \end{split}
   \]
for all $t$ sufficiently large. Thus we have~\ref{cond-1b-aux}).

By~\eqref{eq:d-Lambda-}, since $t \mapsto \Lambda(t)$ is bounded, we get for all $t$ sufficiently large
   \[
   \begin{split}
   x(t)
   & = \e^{-\int_{t_0}^t d(s) \ds} x_0 +\dint_{t_0}^{t}
   \e^{-\int_u^t d(s)\ds} \Lambda(u) \du\\
   & < x_0  + \Lambda_S \dint_{t_0}^t \exp\left(-\frac{t-u}{\omega_d}-1\right)\frac{d^-}{2} \du \\
   & < x_0+\frac{2\omega_d}{d^-} \e^{d^-/2} \Lambda_S,
   \end{split}
   \]
and we conclude that each solution $x(t)$ of~\eqref{eq:SistemaAuxiliar} with initial condition $x(t_0) \ge 0$ is bounded. Additionally, if $x,x_1$ are two solutions of~\eqref{eq:SistemaAuxiliar} with initial condition $x_0=x(t_0) \ge 0$ and $x_{1,0}=x_1(t_0) \ge 0$ we have
   $$|x(t)-x_1(t)| \le \e^{-\int_{t_0}^t d(s) \ds} | x_0 - x_{1,0}|,$$
and we get~\ref{cond-2-aux}).

Finally, if $x(t)$ is a solution of~\eqref{eq:SistemaAuxiliar} and $\tilde x(t)$ is a solution of~\eqref{eq:sist-aux-f} with $\tilde x(t_0)=x(t_0)$ then, setting $u(t)=\tilde x(t)-x(t)$, we obtain the problem
   $$u'=-d(t)u+f(t),$$
with $u(t_0)=0$. The above problem has solution
   \[
\mathbb{}   \begin{split}
   u(t)
   & = \dint_{t_0}^t \e^{-\int_u^t d(s)\ds} f(u) \du \\
   & \le \sup_{t \ge t_0} |f(t)| \dint_{t_0}^t \e^{-\int_u^t d(s)\ds} \du \\
   & \le D \sup_{t \ge t_0} |f(t)|
   \end{split}
   \]
where $D=\frac{2\omega_d}{d^-} \exp\left(\frac{d^-}{2}\right)$. Thus, we obtain~\ref{cond-3-aux}).
\end{proof}

\end{enumerate}
\end{proposition}

For each $\theta,\delta \in \R_0^+$ with $\delta > \theta$ define the set
$$\Delta_{\theta,\delta}=\{(x_1,x_2,x_3,x_4) \in \R^4 \colon \ \theta \le x_1 \le \delta \ \wedge \ 0 \le x_i \le \delta, \ i=2,3,4 \}.$$
We note that every solution $(S(t),I(t),Q(t),R(t))$ of our system stays in the region $\Delta_{0,K}$ (where $K$ is givel by~\ref{cond-3-bs}) in Proposition~\ref{subsection:BS}) for every $t \in \R_0^+$ sufficiently large.

We need some additional assumptions about our system.
Assume that:
\begin{enumerate}[$\text{H}$1)]
\item \label{cond-H1} the functions $x \mapsto \phi(t,x,y,w,z)$ are non decreasing and $\phi(t,0,y,w,z) = 0$;
\item \label{cond-H2a} given $\theta>0$ there is $K_\theta>0$ such that $$|\phi(t,x_1,y,w,z) - \phi(t,x_2,y,w,z)| \le K_\theta |x_1-x_2|z,$$
for each $(t,x_1,y,w,z),(t,x_2,y,w,z)\in \R_0^+ \times \Delta_{\theta,K}$;
\item \label{cond-H2b} There is $N>0$ such that, for every $0<\delta \le K$, \, $y,w,z \in (0,\delta)$ and $t \in \R_0^+$ we have
$$\inf_{0 < \tau < 3\delta} \frac{\phi(t,x,0,0,\tau)}{\tau}
\le \dfrac{\phi(t,x,y,w,z)}{z} < N. $$
\item \label{cond-H4} For every $y,w,z \in (0,K)$ we have
$$\dfrac{\phi(t,x,y,w,z)}{z} \le \limsup_{\delta \to \, 0} \frac{\phi(t,x,0,0,\delta)}{\delta} < +\infty. $$
\end{enumerate}
Note that for contact rates that are not dependent on $y$ and $w$,   H\ref{cond-H2b}) is always verified. Note also that, for differentiable incidence $\phi$, H\ref{cond-H2a}) holds if and only if for each $0 < \theta \le K$ there is $K_\theta>0$ such that
   $$\dfrac{\partial \phi}{\partial x} (t,x,y,w,z) \le K_\theta z,$$
for all $(t,x,y,w,z)\in \R_0^+ \times \Delta_{\delta,\theta}$.

Note that the above conditions hold for the usual incidence functions. Note further that condition~H\ref{cond-H1}) is very natural since the incidence shouldn't decrease with the increase of susceptibles and must be $0$ in the absence of susceptibles. Setting $x_1=x$, $x_2=0$ and $z=0$ in~H\ref{cond-H2a}) we conclude that we must have $\phi(t,x,y,w,0)=0$, another reasonable assumption that states that in the absence of infectives the incidence vanishes.

Conditions~H\ref{cond-H2a}), H\ref{cond-H2b}) and~H\ref{cond-H4}) are not as natural as condition~H\ref{cond-H1}) but somehow we can see them as imposing bounds for the force of infection. Nevertheless, they resemble similar conditions assumed by several authors that studied autonomous epidemic models with general incidence functions. For instance, in~\cite{Driessche-Li-Muldowney-CAMQ-1999}, the authors study a family of autonomous SEIRS models with general incidence functions of the form $\beta g(I) S$ where $g$ is a $C^1$ function verifying $g(0)=0$,
\begin{equation}\label{eq:DLM-CAMQ-1999-cond1}
\lim_{I \to 0^+} \frac{\beta g(I) S}{I} < +\infty
\end{equation}
and
\begin{equation}\label{eq:DLM-CAMQ-1999-cond2}
g(I) \le c I,
\end{equation}
for sufficiently small $I$. Condition~\eqref{eq:DLM-CAMQ-1999-cond1} resembles our condition~H\ref{cond-H4}) and condition~\eqref{eq:DLM-CAMQ-1999-cond2} implies
$$|\beta g(I) S_1 - \beta g(I) S_2| \le \beta g(I) |S_1-S_2| \le \beta c I |S_1-S_2|,$$
corresponding in their situation to our condition~H\ref{cond-H4}).

Now, we need to make some definitions. We say that
\begin{enumerate}[i)]
\item the infectives $I$ are \emph{permanent} in system~\eqref{eq:ProblemaPrincipal} if there exist $0<m_1<m_2$ such that
\[
m_1 < \liminf_{t \to \infty} I(t) \le \limsup_{t \to \infty} I(t) < m_2,
\]
for every solution $(S(t),I(t),R(t),V(t))$ of~\eqref{eq:ProblemaPrincipal} with $S_0(t),I_0(t),R_0(t),V_0(t)>0$ (note that we require that $m_1$ and $m_2$ are independent of the given solution with positive initial condition in $I$);
\item the infectives $I$ go to \emph{extinction} in system~\eqref{eq:ProblemaPrincipal} if $\displaystyle \lim_{t \to \infty} I(t) = 0$ for all solutions of~\eqref{eq:ProblemaPrincipal}.
\end{enumerate}

\section{Statment of the results} \label{section:SR}

For each solution $x(t)$ of~\eqref{eq:SistemaAuxiliar} with $x(0)>0$, define the function
\[
b_\delta(t,x(t))=
\frac{\phi(t,x(t),0,0,\delta)}{\delta} -\left( \gamma(t)+\sigma(t)+d(t)+\alpha_1(t)\right)
\]
and the numbers
\begin{equation}\label{eq:liminf-threshold}
  r_p(\lambda) = \liminf_{t \to \, +\infty}\int_t^{t+\lambda} \liminf_{\delta \to 0} b_\delta(s,x(s)) \ds,
\end{equation}
\begin{equation}\label{eq:limsup-threshold}
  r_e(\lambda) = \limsup_{t \to \, +\infty}\int_t^{t+\lambda}
\limsup_{\delta \to 0} b_\delta(s,x(s)) \ds,
\end{equation}

$$R_p(\lambda)=\e^{r_p(\lambda)} \quad \text{and} \quad R_e(\lambda)=\e^{r_e(\lambda)}.$$
Contrarily to what one could expect, the next technical lemma that will be proved in Section~\ref{section:PT} shows that the numbers above do not depend on the particular solution $x(t)$ of~\eqref{eq:SistemaAuxiliar} with $x(0)>0$.
The numbers $R_p(\lambda)$, $R_e(\lambda)$ play the role of the basic reproduction number (defined for autonomous systems and also called quarantine number) in the non-autonomous setting.

We have the following result whose proof will be given in section~\ref{section:PT}.

\begin{lemma} \label{lema:indep}
We have the following:
\begin{enumerate}
\item Let $\eps > 0$ be sufficiently small and $0 < \theta \le K$. If
$$a,b \in \ (\theta,K) \quad \text{and} \quad |a-b|<\eps,$$
then
\begin{equation} \label{eq:bdela}
|b_\delta(t,a) - b_\delta(t,b)| < \eps K_\theta.
\end{equation}
\item The numbers $R_p(\lambda)$ and $R_e(\lambda)$ are independent of the particular solution $x(t)$ with $x(0)>0$ of~\eqref{eq:SistemaAuxiliar}.
\end{enumerate}
\end{lemma}

Note that, by~\eqref{eq:bdela} we have
\begin{equation} \label{eq:a}
\left|\liminf_{\delta \to 0} b_\delta(t,a) - \liminf_{\delta \to 0} b_\delta(t,b)\right| < \eps K_\theta
\end{equation}
and
\begin{equation} \label{eq:b}
\left|\limsup_{\delta \to 0} b_\delta(t,a) - \limsup_{\delta \to 0} b_\delta(t,b)\right| < \eps K_\theta.
\end{equation}

We now state our main theorem on the permanence and extinction of the infectives in system~\eqref{eq:ProblemaPrincipal}. A proof of this result will be given in section~\ref{section:PT}.

\begin{theorem}\label{teo:Main}
We have the following for the system~\eqref{eq:ProblemaPrincipal}.
\begin{enumerate}
\item \label{teo:Permanence} If there is a constant $\lambda>0$ such that $R_p(\lambda)>1$ then the infectives $I$ are permanent.
\item \label{teo:Extinction} If there is a constant $\lambda>0$ such that $R_e(\lambda)<1$ then the infectives $I$ go to extinction and any disease free solution $(S_0(t),0,0,0)$ is globally attractive.
\end{enumerate}
\end{theorem}

\section{Autonomous model} \label{section:RNAS}

In this section we are going to consider the autonomous setting. Namely, we are going to assume in system~\eqref{eq:ProblemaPrincipal} that $\Lambda, d, \gamma, \sigma, \alpha_1, \epsilon$ and $\alpha_2$ are constant functions, that $\Lambda,d>0$, and that $\phi$ is independent of $t$. We obtain the autonomous system
\begin{equation}\label{eq:autonomo}
\begin{cases}
S'=\Lambda-\phi_0(S,R,Q,I)-d S \\
I'=\phi_0(S,R,Q,I)-\left[\gamma+\sigma+d+\alpha_1 \right]I \\
Q'=\sigma I -[d+\alpha_2+\eps]Q \\
R'=\gamma I + \eps Q -d R
\end{cases}
\end{equation}
where $\phi_0:(\R^+_0)^4 \to \R$ is continuous, nonnegative and satisfies~H\ref{cond-H1}), H\ref{cond-H2a}) and H\ref{cond-H2b}). In this setting we have that the auxiliary equation~\eqref{eq:SistemaAuxiliar} admits the constant solution $x(t)=\Lambda/d$. Thus we obtain
$$r_p(\lambda)=\left[ \liminf_{\delta \to 0} \frac{\phi_0(\Lambda/d,0,0,\delta)}{\delta} - (\gamma+\sigma+d+\alpha_1) \right] \lambda, \quad R_p(\lambda)=\e^{r_p(\lambda)}$$
and
$$r_e(\lambda)=\left[ \limsup_{\delta \to 0} \frac{\phi_0(\Lambda/d,0,0,\delta)}{\delta} - (\gamma+\sigma+d+\alpha_1) \right] \lambda, \quad R_{e,1}(\lambda)=\e^{r_e(\lambda)}.$$
It is now easy to establish a result that is a version of Theorem~\ref{teo:Main} in the particular case of autonomous systems. Define
$$
R^{\aut}_p= \liminf_{\delta \to 0} \frac{\phi_0(\Lambda/d,0,0,\delta)}{\delta(\gamma+\sigma+d+\alpha_1)}
\quad \text{and} \quad
R^{\aut}_e= \limsup_{\delta \to 0} \frac{\phi_0(\Lambda/d,0,0,\delta)}{\delta(\gamma+\sigma+d+\alpha_1)}.
$$
For any $\lambda>0$, we have $R_p(\lambda)>1$ if and only if $R^{\aut}_p>1$ and $R_e(\lambda)<1$ if and only if $R^{\aut}_e<1$ and this implies the following result.

\begin{theorem}\label{teo:autonomous}
We have the following for the autonomous system~\eqref{eq:autonomo}.
\begin{enumerate}
\item If $R^{\aut}_p>1$ then the infectives $I$ are permanent;
\item If $R^{\aut}_e<1$ then the infectives $I$ go to extinction and any disease free solution is globally attractive.
\end{enumerate}
\end{theorem}

Note that, if the incidence has the particular form
\begin{equation}\label{eq:form}
\phi_0(S,R,Q,I)=\psi(S,R,Q)g(I)I,
\end{equation}
where $g$ is a continuous, bounded and nonnegative function,
then
$$R^{\aut}_p=\frac{\psi(\Lambda/d,0,0) \displaystyle \liminf_{\delta \to 0} g(\delta)}{\gamma+\sigma+d+\alpha_1} \quad \text{and} \quad R^{\aut}_e=\frac{\psi(\Lambda/d,0,0) \displaystyle \limsup_{\delta \to 0} g(\delta)}{\gamma+\sigma+d+\alpha_1}$$
and we conclude that the infectives are permanent (resp. go to extinction) if
$$\liminf_{\delta \to 0} g(\delta)>
\frac{\gamma+\sigma+d+\alpha_1}{ \psi(\Lambda/d,0,0)} \quad \text{(resp.} \ \limsup_{\delta \to 0} g(\delta)<
\frac{\gamma+\sigma+d+\alpha_1}{\psi(\Lambda/d,0,0)}\text{)}.$$
Naturally, if
$$\frac{\gamma+\sigma+d+\alpha_1}{\psi(\Lambda/d,0,0)} \in \ \left[\liminf_{\delta \to 0} g(\delta), \ \limsup_{\delta \to 0} g(\delta) \right]$$
we get no information about the asymptotic behavior of the infectives.

Assuming now that $g$ is constant, we have $R^{\aut}_p=R^{\aut}_e$ and we can recover the results obtained in~\cite{Hethcote-Zhien-Shengbing-MB-2002}.
In fact, the autonomous SIQR model with mass-action incidence ($\phi_0(S,R,Q,I)=\beta S I$), the autonomous SIQR model with standard incidence ($\phi_0(S,R,Q,I)=\beta SI / (S+I+Q+R)$) and the autonomous SIQR model with quarantine-adjusted incidence ($\phi_0(S,R,Q,I)=\beta SI / (S+I+R)$) are all in the conditions we have assumed and our threshold values for those model are $R^{\aut}_p=R^{\aut}_e=\beta\Lambda/[d(\gamma+\sigma+d+\alpha_1)]$ for mass-action incidence and $R^{\aut}_p=R^{\aut}_e=\beta/(\gamma+\sigma+d+\alpha_1)$ for standard and quarantine-adjusted incidences. For these models, the numbers $R_p^{\aut}=R_e^{\aut}$ coincide with the quarantine reproduction number obtained in~\cite{Hethcote-Zhien-Shengbing-MB-2002}. Note that, to put the referred autonomous models considered in~\cite{Hethcote-Zhien-Shengbing-MB-2002} in our context, we need to have the standard incidence and the quarantine-adjusted incidence continuous in $\Delta_{\delta,0}$. This constitutes no problem since the functions $\phi_0$ associated to these incidences can be extended to a continuous function in $\Delta_{\delta,0}$, setting $\phi_0(0,R,0,0)=0$ for the standard incidence and $\phi_0(0,0,0,0)=0$ for the quarantine-adjusted incidence.

Set $\psi(S,Q,R)=S$ and $g(I)=\frac{I^{p-1}}{1+\alpha I^q}$ with $p,q>0$, $\alpha \ge 0$, in~\eqref{eq:form}. This family of contact rates was considered for instance in~\cite{Liu-Hethcote-Levin-JMB-1987, Hethcote-denDriessche-JMB-1991}. Our threshold conditions show that if $p<1$ then the disease is permanent (independently of the other parameters) and if $p=1$, setting
 $$R=\frac{\psi(\Lambda/d,0,0)}{\gamma+\sigma+d+\alpha_1},$$
the disease is permanent if $R>1$ and goes to extinction if $R<1$.
\section{Asymptotically autonomous model} \label{section:NSIQRSIR}

In this section we are going to consider the asymptotically autonomous SIQR model. That is, in addition to the assumptions on Theorem~\ref{teo:Main}, we are going to assume for system~\eqref{eq:ProblemaPrincipal} that there is a continuous function $\phi_0$ such that
$$\lim_{t \to \infty} \phi(t,S,R,Q,I) = \phi_0(S,R,Q,I)$$
for each $(S,R,Q,I) \in \R^4$ and that the time-dependent parameters are asymptotically constant: $\beta(t) \to \beta$, $\Lambda(t) \to \Lambda$, $d(t)\to d$, $\gamma(t) \to \gamma$, $\sigma(t) \to \sigma$, $\alpha_1(t) \to \alpha_1$, $\alpha_2(t) \to \alpha_2$ and $\epsilon(t) \to \epsilon$ as $t \to +\infty$. Denoting by $F(t,S,R,Q,I)$ the right hand side of~\eqref{eq:ProblemaPrincipal} and by $F_0(S,R,Q,I)$ the right hand side of the limiting system, that is of~\eqref{eq:autonomo}, we also need to assume that
	$$\lim_{t \to +\infty} F(t,S,R,Q,I) =F_0(S,R,Q,I),$$
with uniform convergence on every compact set of $(\R_0^+)^4$, and that $(S,R,Q,I) \mapsto F(t,S,R,Q,I)$ and $(S,R,Q,I) \mapsto F_0(S,R,Q,I)$ are locally Lipschitz functions.

There is a general setting that will allow us to study this case. Let $f:\R \times R^n \to \R$ and $f_0:\R^n \to \R$ be continuous and locally Lipschitz in $\R^n$. Assume also that the nonautonomous system
\begin{equation}\label{eq:nonaut}
x'=f(t,x)
\end{equation}
is asymptotically autonomous with limit equation
\begin{equation}\label{eq:aut}
x'=f_0(x),
\end{equation}
that is, assume that $f(t,x) \to f_0(x)$ as $t \to +\infty$ with uniform convergence in every compact set of $\R^n$.
The following theorem is a particular case of a result established in~\cite{Markus-CTNO-1956} (for related results and applications see for example~\cite{Chavez-Thieme-MPD-1995, Mischaikow-Smith-Thieme-TAMS-1995}).
\begin{theorem}\label{Markus}
Let $\Phi(t,t_0,x_0)$ and $\phi(t,t_0,y_0)$ be solutions of~\eqref{eq:nonaut} and~\eqref{eq:aut} respectively. Suppose that $e \in \R^n$ is a locally stable equilibrium point of~\eqref{eq:aut} with attractive region
$$W(e)=\left\{y \in \R^n: \lim_{t \to +\infty} \phi(t,t_0,y) = e \right\}$$
and that $W_\Phi \cap W(e) \ne \emptyset$, where $W_\Phi$ denotes the omega limit of $\Phi(t,t_0,x_0)$. Then $\displaystyle \lim_{t \to +\infty} \Phi(t,t_0,x_0)=e$.
\end{theorem}

Since $(\R^+)^4$ is the attractive region for any solution of system~\eqref{eq:autonomo} with initial condition in $(\R^+)^4$ and the omega limit of every orbit of the asymptotically autonomous system with $I(t_0)>0$ is contained in $(\R^+)^4$, we can use Theorem~\ref{teo:autonomous} and Theorem~\ref{Markus} to obtain the following result.

\begin{theorem}
We have the following for the asymptotically autonomous systems above.
\begin{enumerate}\label{teo:asymp-aut}
\item If $R^{\aut}_p>1$ then the infectives $I$ are permanent;
\item If $R^{\aut}_e<1$ then the infectives $I$ go to extinction and any disease free solution is globally attractive.
\end{enumerate}
\end{theorem}

\section{Periodic model with constant natural death and recruitment} \label{section:PER}
In this section we are going to consider a periodic SIQR model. In addition to our assumptions on the function $\phi$ and the parameter functions in Theorem~\ref{teo:Main}, we are going to assume in system~\eqref{eq:ProblemaPrincipal} that $\Lambda(t)=\Lambda$ and $d(t)=d$ are constant functions, that there is a $T>0$ such that $\phi(t,S,R,Q,I)=\phi(t+T,S,R,Q,I)$ and that the remaining time-dependent parameter functions are periodic functions with period $T$.
We have in this case the constant solution $x(t)=\Lambda/d$ and therefore
\[
\begin{split}
r_p(\lambda)
& =\liminf_{t \to +\infty} \int_t^{t+T} \liminf_{\delta \to 0} \frac{\phi(s,\Lambda/d,0,0,\delta)}{\delta} - (\gamma(s)+\sigma(s)+d+\alpha_1(s)) \ds \\
& = \int_0^T \liminf_{\delta \to 0} \frac{\phi(s,\Lambda/d,0,0,\delta)}{\delta} \ds - ( \bar\gamma + \bar\sigma + d + \bar\alpha_1) T
\end{split}
\]
and similarly
$$r_e(\lambda)= \int_0^T \limsup_{\delta \to 0} \frac{\phi(s,\Lambda/d,0,0,\delta)}{\delta} \ds - ( \bar\gamma + \bar\sigma + d + \bar\alpha_1) T,$$
where $\bar f$ denotes the average of $f$ in the interval [0,T]: $\bar f = \frac{1}{T} \int_0^T f(s) \ds$.
Define
\[
R^{\per}_p(\lambda) =  \frac{1}{T} \int_0^T \liminf_{\delta \to 0} \frac{\phi(s,\Lambda/d,0,0,\delta)}{\delta (\bar\gamma + \bar\sigma + d +\bar\alpha_1)} \ds
\]
and
\[
R^{\per}_e(\lambda) =  \frac{1}{T} \int_0^T \limsup_{\delta \to 0} \frac{\phi(s,\Lambda/d,0,0,\tau)}{\tau (\bar\gamma + \bar\sigma + d +\bar\alpha_1)} \ds.
\]

\begin{theorem}\label{teo:nonaut-PER}
We have the following for the periodic system with constant recruitment and death rates.
\begin{enumerate}
\item If $R^{\per}_p>1$ then the infectives $I$ are permanent;
\item If $R^{\per}_e<1$ then the infectives $I$ go to extinction and any disease free solution is globally attractive.
\end{enumerate}
\end{theorem}

If we assume in the periodic model above that the incidence has the particular form $\phi(t,S,R,Q,I)=\beta(t)SI$ with $\beta(t+T)=\beta(t)$, then we obtain
\begin{equation}\label{eq:thresholdperiodic1}
R^{\per}_d(\lambda)=R^{\per}_e(\lambda) = \frac{\bar\beta \Lambda}{d(\bar\gamma + \bar\sigma + d +\bar\alpha_1)}
\end{equation}
and if we assume in that model that $\phi(t,S,R,Q,I)=\beta(t)SI/(S+R+Q+I)$ or that $\phi(t,S,R,Q,I)=\beta(t)SI/(S+R+I)$ with $\beta(t+T)=\beta(t)$ we get
\begin{equation}\label{eq:thresholdperiodic2}
R^{\per}_d(\lambda)=R^{\per}_e(\lambda) = \frac{\bar\beta}{\bar\gamma + \bar\sigma + d +\bar\alpha_1}.
\end{equation}
If the parameter functions are all constant, by~\eqref{eq:thresholdperiodic1} and~\eqref{eq:thresholdperiodic2} we obtain again the thresholds for the autonomous system with mass-action, standard and quarantine-adjusted incidence.

The thresholds~\eqref{eq:thresholdperiodic1} and~\eqref{eq:thresholdperiodic2}
can also be obtained using the methods developed in~\cite{Wang-Zhao-JDDE-2008, Rebelo-Margheri-Bacaer-JMB-2012} for general periodic epidemic models, that constitute periodic versions of the general autonomous model considered in~\cite{Driessche-Watmough-MB-2002}.

\section{General model with mass-action incidence} \label{section:GEN}
In this section we are going to consider a nonautonomous SIQR model with mass-action incidence. In addition to our assumptions on the function $\phi$ and the parameter functions in Theorem~\ref{teo:Main}, we are going to consider the particular cases of mass-action incidence and we will also assume that natural death and recruitment are constant. We have in this case the constant solution $x(t)=\Lambda/d$ for system~\eqref{eq:SistemaAuxiliar}.

For mass-action incidence, $\phi(t,S,I,Q,R)=\beta(t)SI$, we have
\[
\begin{split}
r_p(\lambda)
& =\liminf_{t \to +\infty} \int_t^{t+\lambda} \beta(s)\frac{\Lambda}{d} - (\gamma(s)+\sigma(s)+d+\alpha_1(s)) \ds \\
& \le \lambda \left( \frac{\Lambda}{d} \beta^-_\lambda - [\gamma^-_\lambda+\sigma^-_\lambda+d+(\alpha_1)^-_\lambda] \right)
\end{split}
\]
and similarly
$$r_e(\lambda) \ge \lambda \left( \frac{\Lambda}{d} \beta^+_\lambda - [\gamma^+_\lambda+\sigma^+_\lambda+d+(\alpha_1)^+_\lambda] \right).$$
Define
\[
R^{SIM}_p(\lambda) =  \frac{\Lambda \beta^-_\lambda}
{d[\gamma^-_\lambda+\sigma^-_\lambda+d+(\alpha_1)^-_\lambda]}
\]
and
\[
R^{SIM}_p(\lambda) =  \frac{\Lambda \beta^+_\lambda} {d[\gamma^+_\lambda+\sigma^+_\lambda+d+(\alpha_1)^+_\lambda]}.
\]

\begin{theorem}
We have the following for the general non-autonomous system with mass-action incidence and constant recruitment and death rates.
\begin{enumerate}\label{teo:nonaut-SIM}
\item If there is $\lambda>0$ such that $R^{SIM}_p(\lambda)>1$ then the infectives $I$ are permanent;
\item If there is $\lambda>0$ such that $R^{SIM}_e(\lambda)<1$ then the infectives $I$ go to extinction and any disease free solution is globally attractive.
\end{enumerate}
\end{theorem}

\section{General model with standard and quarantine-adjusted incidence} \label{section:GENSQA}
Now, we will consider the general nonautonomous SIQR models with standard and quarantine-adjusted incidence. In addition to our assumptions on the function $\phi$ and the parameter functions in Theorem~\ref{teo:Main}, we are going to consider the particular cases of standard and quarantine-adjusted incidence. Let $x(t)$ denote a solution of~\eqref{eq:SistemaAuxiliar} with $x(0)>0$.

For standard and quarantine-adjusted incidence, respectively 	$$\phi(t,S,I,Q,R)=\frac{\beta(t)SI}{S+I+R+Q} \quad \text{ and } \quad \phi(t,S,I,Q,R)=\frac{\beta(t)SI}{S+I+R},$$
we have
\[
\begin{split}
r_p(\lambda)
& =\liminf_{t \to +\infty} \int_t^{t+\lambda} \beta(s) \ \liminf_{\delta \to 0} \frac{x(s)}{x(s)+\delta} - (\gamma(s)+\sigma(s)+d+\alpha_1(s)) \ds \\
& \le \lambda \left( \beta^-_\lambda - [\gamma^-_\lambda+\sigma^-_\lambda+d+(\alpha_1)^-_\lambda] \right)
\end{split}
\]
and similarly
$$r_e(\lambda) \ge \lambda \left( \beta^+_\lambda - [\gamma^+_\lambda+\sigma^+_\lambda+d+(\alpha_1)^+_\lambda] \right).$$
Define
\[
R^{S/QA}_p(\lambda) =  \frac{\beta^-_\lambda}
{\gamma^-_\lambda+\sigma^-_\lambda+d+(\alpha_1)^-_\lambda}
\]
and
\[
R^{S/QA}_e(\lambda) =  \frac{\beta^+_\lambda} {\gamma^+_\lambda+\sigma^+_\lambda+d+(\alpha_1)^+_\lambda}.
\]

\begin{theorem}\label{teo:nonaut-S-QA}
We have the following for the general non-autonomous system with standard or quarantine-adjusted incidence.
\begin{enumerate}
\item If there is $\lambda>0$ such that $R^{S/QA}_p(\lambda)>1$ then the infectives $I$ are permanent;
\item If there is $\lambda>0$ such that $R^{S/QA}_e(\lambda)<1$ then the infectives $I$ go to extinction and any disease free solution is globally attractive.
\end{enumerate}
\end{theorem}
\section{Proof of Theorem~\ref{teo:Main}} \label{section:PT}

\subsection{Proof of Lemma~\ref{lema:indep}}
Assume that $\eps>0$, that $0<\theta\le K$, that $a,b \in (\theta, \delta)$ and that $|a-b|<\eps$. We have, by~H\ref{cond-H2a}), for every $0 \le \tau \le \delta$, $$|\phi(t,a,0,0,\delta) - \phi(t,b,0,0,\delta)| \le K_\theta |a-b| \delta < K_\theta \delta \eps.$$
Therefore,
\begin{equation} \label{eq:supinftau}
\dfrac{\phi(t,a,0,0,\delta)}{\delta} - K_\theta \eps < \dfrac{\phi(t,b,0,0,\delta)}{\delta} < \dfrac{\phi(t,a,0,0,\delta)}{\delta}+K_\theta \eps,
\end{equation}
and, adding and subtracting $\gamma(t)+\sigma(t)+d(t)+\alpha_1(t)$, we get~\eqref{eq:bdela}.

We will now show that in fact $r_p(\lambda)$, $r_e(\lambda)$ are independent of the particular solution $x(t)$ of~\eqref{eq:SistemaAuxiliar} with $x(0)>0$. In fact, by~\ref{cond-2-aux}) in Proposition~\ref{eq:subsectionAS}, for every $\eps>0$ and every solution $x_1(t)$ of~\eqref{eq:SistemaAuxiliar} with $x_1(0)>0$ there is a $T_\eps >0$ such that $|x(t)- x_1(t)| < \eps$ for every $t \ge T_\eps$. Choose $\eps_K>0$ such that, for $\eps < \eps_K$ and $t \ge T_\eps$, we have $x(t), x_1(t), x(t) \pm \eps \in \, (\theta,K)$, for some $\theta>0$ that depends on $x$ and $x_1$.
Leting $a=x(t)$, $b=x_1(t)$ we obtain by~\eqref{eq:a}
\[
\begin{split}
\int_t^{t+\lambda} \liminf_{\delta \to 0} b_\delta(s,x_1(s)) \ds - \lambda K_\theta \eps
& < \int_t^{t+\lambda} \liminf_{\delta \to 0} b_\delta(s,x(s)) \ds \\
& < \int_t^{t+\lambda} \liminf_{\delta \to 0} b_\delta(s,x_1(s)) \ds + \lambda K_\theta \eps,
\end{split}
\]
for every $t \ge T_\eps$. We conclude that, for every $0<\eps<\eps_K$,
\[
\left| \liminf_{t \to +\infty} \int_t^{t+\lambda} \liminf_{\delta \to 0} b_\delta(s,x(s)) \ds - \liminf_{t \to +\infty} \int_t^{t+\lambda} \ \liminf_{\delta \to 0} b_\delta(s,x_1(s)) \ds \right|  < \lambda K_\theta \eps,
\]
and thus
\[
\liminf_{t \to +\infty} \int_t^{t+\lambda} \liminf_{\delta \to 0} b_\delta(s,x(s)) \ds = \liminf_{t \to +\infty} \int_t^{t+\lambda} \liminf_{\delta \to 0} b_\delta(s,x_1(s)) \ds.
\]

Letting $\delta \to 0$ we conclude that $r_p(\lambda)$ (and thus $R_p(\lambda)$) is independent of the chosen solution. Using~\eqref{eq:b}, the same reasoning shows that $r_e(\lambda)$ (and thus $R_e(\lambda)$) are also independent of the particular solution. This proves the lemma.

\subsection{Proof of Theorem~\ref{teo:Main}}
Assume that $R_p(\lambda)>1$ (and thus $r_p(\lambda)>0$) and let $(S(t),I(t),R(t),Q(t))$ be any solution of~\eqref{eq:ProblemaPrincipal}
with $S(T_0)>0$, $I(T_0)>0$, $R(T_0)>0$ and $Q(T_0)>0$ for some $T_0 \ge 0$. By~\ref{cond-3-bs}) in Proposition~\ref{subsection:BS} we may assume that for $t \ge T_0$ we have $(S(t),I(t),R(t),Q(t)) \in \Delta_{0,K}$.

Since $r_p(\lambda)>0$, using~\eqref{eq:a} we conclude that there are constants $0 < \delta_1 < K$ and $\Theta >0$ such that
\begin{equation}\label{eq:ByHypothesis}
\int_t^{t+\lambda} \inf_{0<\delta<3\delta_1} b_\delta(s,x^*(s)-\eps) \ds > \Theta
\end{equation}
for all $t \ge 0$ sufficiently large, say $t \ge T_1$, and $\eps>0$ sufficiently small, say for $0<\eps\le \bar \eps$, where $x^*$ is any fixed solution of~\eqref{eq:SistemaAuxiliar} with $x^*(0)>0$. Note that, by~\ref{cond-1a-aux}) in Proposition~\ref{eq:subsectionAS} we have $x^*(T_0)>0$ and by by~\ref{cond-1b-aux}) in Proposition~\ref{eq:subsectionAS} we may assume (eventually using a bigger $T_1$) that $x^*(t) \ge \theta$ for some $\theta >0$ and $t \ge T_1$.

Define
\begin{equation}\label{eq:variacao_eps}
\eps_0 = \min \left\{ \frac{\Theta}{4K_\theta\lambda}, \bar \eps \right\} \quad \text{and} \quad \eps_1 = \min\left\{\delta_1,\frac{\delta_1}{C},\frac{\eps_0}{2DN}\right\},
\end{equation}
where $K_\theta$ is given by H\ref{cond-H2a}), $N$ is given by H\ref{cond-H2b}), $C$ is given by~\ref{cond-4-bs}) in Proposition~\ref{subsection:BS} and $D$ is given by~\ref{cond-3-aux}) in Proposition~\ref{eq:subsectionAS}.

We will show that
\begin{equation} \label{eq:limsup-I-ge-0}
\limsup_{t \to +\infty} I(t) \ge \eps_1.
\end{equation}
Assume by contradiction that~\eqref{eq:limsup-I-ge-0} is not true. Then there exists $T_2 \ge 0$ satisfying $I(t) < \eps_1$ for all $t \ge T_2$. Consider the auxiliary equation
\begin{equation}\label{eq:SistemaAuxiliar2}
x'=\Lambda(t) - d(t) x - N \eps_1,
\end{equation}
where $N$ is given by~H\ref{cond-H2b}).

Let $\bar x(t)$ be the solution of~\eqref{eq:SistemaAuxiliar} with $\bar x(T_0)=S(T_0)$ and let $x(t)$ be the solution of~\eqref{eq:SistemaAuxiliar2} with $x(T_0)=S(T_0)$.  By~\eqref{eq:variacao_eps} and \ref{cond-3-aux}) in Proposition~\ref{eq:subsectionAS} we obtain
\begin{equation}\label{eq:meio-eps-0}
|x(t) - \bar x(t)| \le D N \eps_1 \le \frac{\eps_0}{2},
\end{equation}
for all $t \ge T_0$.

According to~\ref{cond-2-aux}) in Proposition~\ref{eq:subsectionAS}, $x^*(t)$
is globally uniformly attractive on $\R_0^+$. Therefore, there exists $T_3>0$ such that, for all $t \ge T_3$, we have
\begin{equation}\label{eq:xbar-x*}
|\bar x(t) - x^*(t)| \le \frac{\eps_0}{2}.
\end{equation}

By~H\ref{cond-H2b}) we have
$$0 \le \phi(t,S(t),R(t),Q(t),I(t))\le N I(t) < N \eps_1,$$
for all $t \ge \max\{T_0,T_2\}$, and thus, by the first equation in~\eqref{eq:ProblemaPrincipal},
\begin{equation}\label{eq:SistemaAuxiliar2-Maj}
\Lambda(t) - d(t) S \ge S' > \Lambda(t) - d(t) S - N \eps_1.
\end{equation}
for all $t \ge \max\{T_0,T_2\}$. Comparing~\eqref{eq:SistemaAuxiliar2}
and~\eqref{eq:SistemaAuxiliar2-Maj} we conclude that $S(t) > x(t)$ for all $t \ge \max\{T_0,T_2\}$.  Write $T_4=\max\{T_0,T_1,T_2,T_3\}$. By~\eqref{eq:meio-eps-0} and~\eqref{eq:xbar-x*}
we have, for all $t \ge T_4$,
$$x^*(t)+\eps_0 > x^*(t)+\frac{\eps_0}{2} \ge \bar{x}(t) \ge S(t) > x(t) \ge \bar x(t) -\frac{\eps_0}{2} \ge x^*(t) - \eps_0,$$
and thus
\begin{equation} \label{eq:perm1}
|S(t)-x^*(t)+\eps_0| \le |S(t)-x^*(t)|+\eps_0 <2\eps_0,
\end{equation}
for all $t \ge T_4$.
By~\eqref{eq:perm1} and~\eqref{eq:a} in Lemma~\ref{lema:indep} we get, for all $t \ge T_4$,
\begin{equation} \label{eq:perm2}
b_\delta(t, S(t)) > b_\delta(t,x^*(t)-\eps_0) - 2K_\theta \eps_0.
\end{equation}

For $t > 0$ sufficiently large, say $t \ge T_5$, we have by~\eqref{eq:boundQR} and~\eqref{eq:variacao_eps}, $Q(t)<C\eps_1<\delta_1$ and $R(t)<C\eps_1<\delta_1$ and thus $I(t),Q(t),R(t) \in (0,\delta_1)$. We may assume that $T_5 \ge T_4$.

According to~\eqref{eq:variacao_eps} we have $\eps_0 \le \bar\eps$ and by~\eqref{eq:ByHypothesis} we get
\begin{equation}\label{eq:majinteg}
\int_{T_5}^t \inf_{0<\delta<3\delta_1} b_\delta(\tau,x^*(s)-\eps_0) \ds > \left( \frac{1}{\lambda} (t-T_5) -1 \right) \Theta.
\end{equation}

By the second equation in~\eqref{eq:ProblemaPrincipal} we obtain
$$ I'(t)=\left[\frac{\phi(t,S(t),R(t),Q(t),I(t))}{I(t)}-(\gamma(t)+\sigma(t)+d(t)+\alpha_1(t))\right] I(t)
$$
and thus, integrating from $T_5$ to $t$, recalling that $I(t),Q(t),R(t) \in \ ]0,\delta_1[$ for all $t \ge T_5$ and using~H\ref{cond-H2b}), \eqref{eq:perm2}, \eqref{eq:majinteg} and~\eqref{eq:variacao_eps}, we have
\[
\begin{split}
I(t)
& = I(T_5) \e^{\int_{T_5}^t \phi(s,S(s),R(s),Q(s),I(s)) / I(s) -(\gamma(s)+\delta(s)+d(s)+\alpha_1(s)) \ds} \\
& \ge I(T_5) \e^{\int_{T_5}^t \underset{0<\delta<3\delta_1}{\inf} b_\delta(s, S(s)) \ds} \\
& > I(T_5) \e^{\int_{T_5}^t \underset{0<\delta<3\delta_1}{\inf}  b_\delta(s,x^*(s)-\eps_0) \ds - 2K_\theta \eps_0(t-T_5)}\\
& > I(T_5) \e^{(\frac{\Theta}{\lambda} - 2K_\theta \eps_0)(t-T_5) -\Theta} \\
& \ge I(T_5) \e^{\frac{\Theta}{2\lambda}(t-T_5) -\Theta},
\end{split}
\]
and we conclude that $I(t) \to +\infty$. This contradicts the assumption
that $I(t) < \eps_1$ for all $t \ge T_2$. From this we conclude that~\eqref{eq:limsup-I-ge-0} holds.

Next we will prove that for some constant $\ell>0$ we have in fact
\begin{equation}\label{eq:liminf-ge-0}
\liminf_{t \to +\infty} I(t)  > \ell
\end{equation}
for every solution $(S(t),I(t),R(t),Q(t))$ with $S(T_0)>0$, $I(T_0)>0$, $R(T_0)>0$ and $Q(T_0)>0$. By~\eqref{eq:ByHypothesis}, for all $\xi \ge \lambda$, $t \ge T_1$ and $0<\theta<\bar\eps$, we have
\begin{equation}\label{eq:ByHypothesis2}
\int_t^{t+\xi} \inf_{0<\delta\le\delta_1} b_\delta(s,x^*(s)-\theta) \ds > N.
\end{equation}

We proceed by contradiction. Assume that~\eqref{eq:liminf-ge-0} does not hold. Then there exists a sequence of initial values $(x_n)_{n \in \N}$, with $x_n=(S_n,R_n,Q_n,I_n)$ with $S_n>0$, $R_n>0$, $Q_n>0$ and $I_n>0$ such that
\[
\liminf_{t \to +\infty} I(t,x_n) < \frac{\eps_1}{n^2},
\]
where $I(t,x_n)$ denotes the solution of~\eqref{eq:ProblemaPrincipal} with initial conditions $S(T_1)=S_n$, $I(T_1)=I_n$, $R(T_1)=R_n$ and $Q(T_1)=Q_n$.
By~\eqref{eq:limsup-I-ge-0}, given $n \in \N$, there are two
sequences $(t_{n,k})_{k \in \N}$ and $(s_{n,k})_{k \in \N}$ with
\[
T_1 < s_{n,1} < t_{n,1} < s_{n,2} < t_{n,2} < \cdots < s_{n,k} < t_{n,k} < \cdots
\]
and $\displaystyle \lim_{k \to +\infty} s_{n,k}  = +\infty$, such that
\begin{equation}\label{eq:maj-eps-eps1a}
I(s_{k,n},x_n) = \frac{\eps_1}{n}, \quad I(t_{k,n},x_n) = \frac{\eps_1}{n^2}
\end{equation}
and
\begin{equation}\label{eq:maj-eps-eps2}
\frac{\eps_1}{n^2} < I(t,x_n) < \frac{\eps_1}{n}, \quad \text{for all} \quad t \in ]s_{n,k},t_{n,k}[.
\end{equation}
By the second equation in~\eqref{eq:ProblemaPrincipal} we have
\[
\begin{split}
I'(t,x_n)
& = I(t,x_n) [ (\phi(t, S(t,x_n), R(t,x_n), Q(t,x_n), I(t,x_n))/I(t,x_n) - \\
& \quad - (\gamma(t)+\delta(t)+d(t)+\alpha_1(t)) ] \\
& \ge - (\gamma(t)+\sigma(t)+d(t)+\alpha_1(t)) I(t,x_n) \\
& \ge - (\gamma_S+\sigma_S+d_S+(\alpha_1)_S) I(t,x_n),
\end{split}
\]
Therefore we obtain
\[
\int_{s_{k,n}}^{t_{k,n}} \frac{I'(\tau,x_n)}{I(\tau,x_n)} \dtau \ge -(\gamma_S+\sigma_S+d_S+(\alpha_1)_S) (t_{k,n}-s_{k,n})
\]
and thus $ I(t_{k,n},x_n) \ge I(s_{k,n},x_n) \e^{ -(\gamma_S+\sigma_S+d_S+(\alpha_1)_S) (t_{k,n}-s_{k,n})}$.
By~\eqref{eq:maj-eps-eps1a} we get $\frac{1}{n} \ge \e^{-(\gamma_S+\sigma_S+d_S+(\alpha_1)_S) (t_{k,n}-s_{k,n})}$
and therefore we have
\begin{equation}\label{eq:limite-t-s}
t_{k,n}-s_{k,n}  \ge \frac{\log n}{\gamma_S+\sigma_S+d_S+(\alpha_1)_S} \to +\infty
\end{equation}
as $n \to +\infty$.

According to~\eqref{eq:maj-eps-eps2}, for all $t \in ]s_{k,n},t_{k,n}[$, we have $I(t,x_n)< \eps_1 / n < \eps_1$ and we conclude that $\phi(t,S,R,Q,I) \le N \eps_1$ for all $t \in ]s_{k,n},t_{k,n}[$. Thus by~\eqref{eq:ProblemaPrincipal} we obtain
\[
S' \ge \Lambda(t)-d(t)S-N\eps_1
\]
for all $t \in ]s_{k,n},t_{k,n}[$.

By comparison we have $S(t,x_n) \ge x(t)$ for all $t \in ]s_{k,n},t_{k,n}[$, where $x$ is the solution of~\eqref{eq:SistemaAuxiliar2} with $x(s_{k,n})=S(s_{k,n},x_n)$. Using~\eqref{eq:meio-eps-0} we obtain, for all $t \in ]s_{k,n},t_{k,n}[$,
\begin{equation}\label{eq:maj-x-barx-y-bary}
|x(t) - \bar x(t)| \le \frac{\eps_0}{2},
\end{equation}
where $\bar x(t)$ is a solution of~\eqref{eq:SistemaAuxiliar}
with $\bar x(s_{k,n})=S(s_{k,n},x_n)$. Since $x^*(t)$  is globally uniformly attractive, there exists $T^*>0$, independent of $n$ and $k$, such that
\begin{equation}\label{eq:maj-x-x*-y-y*}
|\bar x(t) - x^*(t)| \le \frac{\eps_0}{2},
\end{equation}
for all $t \ge s_{k,n}+T^*$. By~\eqref{eq:limite-t-s} we can choose
$B > 0$  such that $t_{n,k}-s_{n,k} > \lambda + T^*$ for all $n \ge B$. Given $n \ge B$, by~\eqref{eq:ByHypothesis2}, \eqref{eq:maj-x-barx-y-bary}
and~\eqref{eq:maj-x-x*-y-y*} and by the second equation
in~\eqref{eq:ProblemaPrincipal} we get
\[
\begin{split}
\frac{\eps_1}{n^2}
& = I(t_{k,n},x_n) \ge I(s_{k,n}+T^*,x_n) \e^{\int_{s_{k,n}+T^*}^{t_{k,n}} \underset{\delta \to 0}{\liminf} \, b_\delta(\tau,S(\tau,x_n)) \dtau} \\
& \ge \frac{\eps_1}{n^2} \e^{\int_{s_{k,n}+T^*}^{t_{k,n}} \underset{\delta \to 0}{\liminf} \, b_{I,\delta_1}(\tau,x^*(\tau)-\eps_0) \dtau} >
\frac{\eps_1}{n^2}.
\end{split}
\]
This leads to a contradiction and establishes that $\displaystyle \liminf_{t \to +\infty} I(t) > \ell >0$. Thus,~\ref{teo:Permanence}. in Theorem~\ref{teo:Main} is established.

Assume now that $R_e(\lambda)<1$ (and thus $r_e(\lambda)<0$) for some $\lambda>0$. Then we can choose $\Theta>0$ and $T_0>0$ such that
\begin{equation}\label{eq:condextinctionA}
\int_t^{t+\lambda} b_\delta(s,x^*(s)+\eps) \ds < -\Theta,
\end{equation}
for all $t \ge T_0$ and $\eps>0$ sufficiently small, say $\eps<\bar\eps$. By~\ref{cond-3-bs}) in Proposition~\ref{subsection:BS} we can assume that $(S(t),I(t),R(t),Q(t)) \in \Delta_{0,K}$ for $t \ge T_0$.

We will show that
\begin{equation}\label{eq:cotrad2}
  \limsup_{t \to +\infty} I(t)=0.
\end{equation}

By~\eqref{eq:ProblemaPrincipal} and~H\ref{cond-H2b}) we have
\[
S'(t) \le \Lambda(t)-d(t)S
\]
for all $t \ge T_0$. For any solution $x(t)$ of~\eqref{eq:SistemaAuxiliar} with $x(T_0)=S(T_0)$ we have by comparison $S(t) \le x(t)$ for all $t \ge T_0$.
Since $x^*(t)$ is globally uniformly attractive, there exists $\eps_1 < \bar\eps$ and $T_1 \ge T_0$ such that, for all $t>T_1$, we have
\[
|x(t)-x^*(t)| < \eps_1.
\]
Therefore, for all $t \ge T_1$, we have
\begin{equation} \label{eq:extinction-S-x}
S(t) \le x(t) \le x^*(t)+\eps_1.
\end{equation}
Thus, by the second equation in~\eqref{eq:ProblemaPrincipal} and by~H\ref{cond-H4}) we conclude that
\[
\begin{split}
\frac{I'(t)}{I(t)}
& = \frac{\phi(t,S(t),R(t),Q(t),I(t))}{I(t)} - (\gamma(t)+\sigma(t)+d(t)+\alpha_1(t)) \\
& \le \limsup_{\delta \to 0} \frac{\phi(t,S(t),0,0,\delta)}{\delta}  -(\gamma(t)+\sigma(t)+d(t)+\alpha_1(t)) \\
& \le \limsup_{\delta \to 0} b_\delta(t,S(t))
\end{split}
\]
and finally
\[
\log \frac{I(t)}{I(T_1)} = \int^{t}_{T_1} \frac{I'(\tau)}{I(\tau)} \dtau \le
\int^{t}_{T_1} \underset{\delta \to 0}{\limsup} \, b_\delta(s,S(s)) \ds.
\]
Therefore, by~H\ref{cond-H1}), \eqref{eq:extinction-S-x} and~\eqref{eq:b} we have
\[
I(t) = I(T_1) \e^{\int_{T_1}^t \underset{\delta \to 0}{\limsup} \, b_\delta(s,S(s)) \ds}
\le I(T_1) \e^{\int_{T_1}^t \underset{\delta \to 0}{\limsup} \, b_\delta(s, x^*(s)+\eps_1)\ds}.
\]
By~\eqref{eq:condextinctionA} we conclude that
$$I(t) \le I(T_1) \e^{-\left(\frac{t-T_1}{\lambda}+1\right)\Theta}.$$
Therefore $\displaystyle \lim_{t \to +\infty} I(t)=0$ and this proves that the infectives go to extinction.

Next, still assuming that $R_e(\lambda)<1$ for some $\lambda>0$, we let $(S(t),R(t),Q(t),I(t))$ be a solution of~\eqref{eq:ProblemaPrincipal} with non-negative initial conditions and $(S_0(t),0,0,0)$ be a disease-free solution of~\eqref{eq:ProblemaPrincipal} with non-negative initial conditions.
By~\ref{cond-3-bs}) in Proposition~\ref{subsection:BS} we can assume that $(S(t),I(t),R(t),Q(t))$, $(S_0(t),0,0,0) \in \Delta_{0,K}$
for $t \ge T_0$. Thus
$$ S_0(t),S(t),R(t),Q(t) < K,$$
for $t \ge T_0$. By Theorem~\ref{teo:Extinction} we have $\displaystyle \lim_{t \to +\infty} I(t) = 0$. Therefore given $\eps>0$ there exists $T_\eps \ge T_0$ such that $I(t)<\eps$, $S(t),R(t),Q(t) < K$ for $t \ge T_\eps$. Thus
$$ \Lambda(t)-d(t)S \ge S'(t) \ge \Lambda(t)-d(t) S - \eps N,$$
for all $t \ge T_\eps$. By comparison we get
\begin{equation} \label{eq:DFS-comp}
x_1(t) \ge S(t) \ge x_2(t)
\end{equation}
 where
$x_1(t)$ is the solution of~\eqref{eq:SistemaAuxiliar} and $x_2(t)$ is the solution of~\eqref{eq:SistemaAuxiliar2} with $\eps_1=\eps$ and $x_1(T_\eps) = x_2(T_\eps) = S(T_\eps)$. By~\ref{cond-3-aux}) in Proposition~\ref{eq:subsectionAS} we obtain, for all $t \ge T_1$,
\begin{equation}\label{eq:DFS-A}
|x_2(t) - x_1(t)| \le \eps N D,
\end{equation}
Since $S_0(t)$ is a solution of~\eqref{eq:SistemaAuxiliar}, it is globally uniformly attractive, by~\ref{cond-2-aux}) in Proposition~\ref{eq:subsectionAS}. Thus, there is $T_2 \ge T_\eps$ such that
\begin{equation}\label{eq:DFS-B}
|x_1(t) - S_0(t)| \le \eps
\end{equation}
for all $t \ge T_2$. By~\eqref{eq:DFS-comp}, \eqref{eq:DFS-A} and~\eqref{eq:DFS-B} we conclude that
\begin{equation}\label{eq:diseasefree-1}
S(t) \ge x_2(t) \ge x_1(t) - \eps ND \ge
S_0(t) - (1+ND)\eps,
\end{equation}
for all $t \ge T_2$, and also, by~\eqref{eq:DFS-comp} and \eqref{eq:DFS-B},
\begin{equation}\label{eq:diseasefree-2}
S(t) \le S_0(t) + \eps
\end{equation}
for all $t \ge T_2$.
Since $\eps>0$ can be made arbitrarily small by taking $t \ge T_\eps$, by~\eqref{eq:diseasefree-1} and~\eqref{eq:diseasefree-2} we have
\begin{equation}\label{eq:S-to-S0}
\lim_{t \to +\infty} |S(t) - S_0(t)|=0.
\end{equation}

Let $T_3 \ge 0$ be such that $I(t)<\eps$ for every $t \ge T_3$. By the third equation in~\eqref{eq:ProblemaPrincipal}, we get
$$Q'(t) \le \eps \sigma_S - (d_I+(\alpha_2)_I+\eps_I) Q(t).$$
and thus
$$Q(t) \le Q(T_3) \e^{-(d_I+(\alpha_2)_I+\eps_I)(t-T_3)} +\frac{\eps \sigma_S (1-\e^{-(d_I+(\alpha_2)_I+\eps_I)(t-T_3)})}{d_I+(\alpha_2)_I+\eps_I}.$$
Therefore
\begin{equation}\label{eq:Q-to-Q0}
\lim_{t \to +\infty} Q(t) \le \frac{\eps \sigma_S}{d_I+(\alpha_2)_I+\eps_I}.
\end{equation}
Since $\eps>0$ can be made arbitrarily small we get $\displaystyle \lim_{t \to +\infty} Q(t)=0$.

Let $T_4 \ge T_3$ be such that $Q(t)<\eps$ for every $t \ge T_4$. By the fourth equation in~\eqref{eq:ProblemaPrincipal} and since $I(t)<\eps$ and for all $t \ge 3$, we get
$$R'(t) \le \eps (\gamma_S+\eps_S) -d_I R(t).$$
We obtain
$$  R(t) \le R(T_4) \e^{-d_I(t-T_4)}+\frac{ \eps (\gamma_S+\eps_S)(1-\e^{-d_I(t-T_4)})}{d_I}.$$
and thus
\begin{equation}\label{eq:R-to-R0}
\lim_{t \to +\infty} R(t) \le \frac{\eps (\gamma_S+\eps_S)}{d_I}.
\end{equation}
Since $\eps>0$ can be made arbitrarily small we get $\displaystyle \lim_{t \to +\infty} R(t)=0$.

We conclude that any disease-free solution $(S_0(t),0,0,0)$ with nonnegative initial conditions is globally attractive. We obtain \ref{teo:Extinction}. in Theorem~\ref{teo:Main} and this concludes our proof.

\section{Simulation}\label{section:SIM}

In order to illustrate our results we carried out some experiments.

Note that, comparing corollaries~\ref{teo:asymp-aut},~\ref{teo:nonaut-PER},~\ref{teo:nonaut-SIM} and~\ref{teo:nonaut-S-QA} with the corresponding autonomous situations we see that the thresholds for each of these models are identical to the thresholds of an autonomous model with parameters that are average type quantities obtained from the corresponding nonautonomous parameteres. In this section we will confirm this fact for some model but we will also show that this is not always the case. Namely, we will define a set of parameters, including a nonautonomous $\beta(t)$ that depend on some $\alpha>0$.

Consider a nonautonomous model with mass-action and parameters $\Lambda=0.001$, $d=0.035$, $\gamma(t)=0.4$, $\sigma(t)=0.01$, $a_1(t)=a_2(t)=\eps(t)=0.2$ and $$\beta(t)=\alpha(1-0.7\sin(0.3 t))(2-\e^{-t}).$$

Firstly, we will confirm our results by computing the thresholds for the models with the above parameteres and mass-action and also quarantine-adjusted incidence.
For any $\lambda>0$, we obtain
\[
R^{SIM}_p(\lambda)=\liminf_{t\to\infty} \ 0.0886\alpha (\lambda-4.667\sin(.3t+.15\lambda)\sin(.15\lambda))
\]
and
\[
R^{SIM}_e(\lambda)=\limsup_{t\to\infty} \ 0.0886\alpha (\lambda-4.667\sin(.3t+.15\lambda)\sin(.15\lambda))
\]
(where the numbers $R^{SIM}_p(\lambda)$ and $R^{SIM}_e(\lambda)$ are given in section~\ref{section:GEN}). For instance, setting $\alpha=9$ and $\lambda=21$, we have $R^{SIM}_p(\lambda)=1.10599>1$ and we can conclude that the infectives are permanent and, setting $\alpha=8$ and $\lambda=21$, we have $R^{SIM}_e(\lambda)=0.98310<1$ and we can conclude that the infectives go to extinction. Figure~\ref{fig1} illustrate the referred situations.
\begin{figure}[h]
  \begin{minipage}[t][3.5cm]{.49\textwidth}
        \includegraphics[scale=0.58]{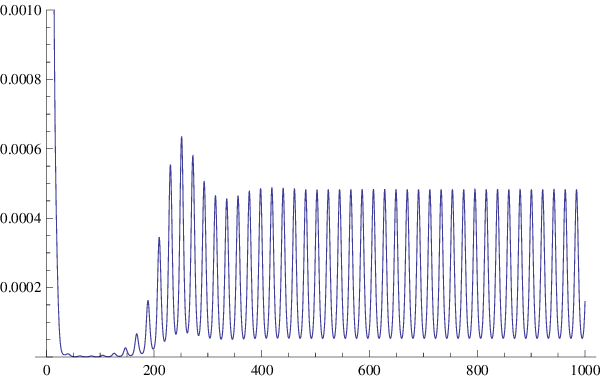}
  \end{minipage}
  \begin{minipage}[t][3.5cm]{.49\textwidth}
        \includegraphics[scale=0.58]{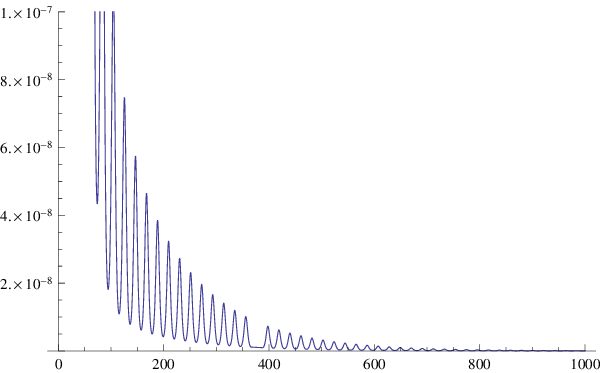}
  \end{minipage}
     \caption{Constant birth and death, mass-action - left $\alpha=9$ ($R^{SIM}_p \approx 1.11$) and right $\alpha=8$ ($R^{SIM}_e\approx 0.98$)}
     \label{fig1}
\end{figure}
Now consider a nonautonomous model with quarantine adjusted incidence and the same parameters: $\Lambda=0.001$, $d=0.035$, $\gamma(t)=0.4$, $\sigma(t)=0.01$, $a_1(t)=a_2(t)=\eps(t)=0.2$ and $\beta(t)=\alpha(1-0.7\sin(0.3 t))(2-\e^{-t})$. For any $\lambda>0$, we obtain in this case
\[
R^{S/QA}_p(\lambda) = \liminf_{t\to\infty} \ 3.10078 \alpha (\lambda-4.667\sin(.3t+.15\lambda)\sin(.15\lambda))
\]
and
\[
R^{S/QA}_e(\lambda) = \limsup_{t\to\infty} \ 3.10078 \alpha (\lambda-4.667\sin(.3t+.15\lambda)\sin(.15\lambda))
\]
(where the numbers $R^{S/Q}_p(\lambda)$ and $R^{S/Q}_p(\lambda)$ are given in section~\ref{section:GEN}). Setting $\alpha=0.25$ and $\lambda=0.2$, we have $R^{S/QA}_p(\lambda)=1.07527>1$ and we can conclude that the infectives are permanent and, setting $\alpha=.23$ and $\lambda=0.2$, we have $R^{S/QA}_e(\lambda)=0.989247<1$ and we can conclude that the infectives go to extinction. Figures~\ref{fig2} illustrate the referred situations.
\begin{figure}[h]
  \begin{minipage}[b][3.5cm]{.49\linewidth}
    \includegraphics[scale=0.58]{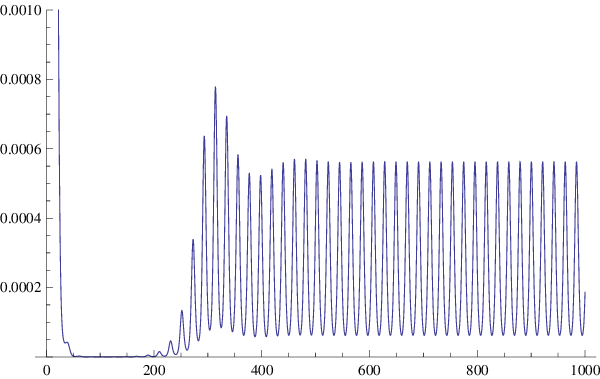}
  \end{minipage}
  \begin{minipage}[b][3.5cm]{.49\linewidth}
        \includegraphics[scale=0.58]{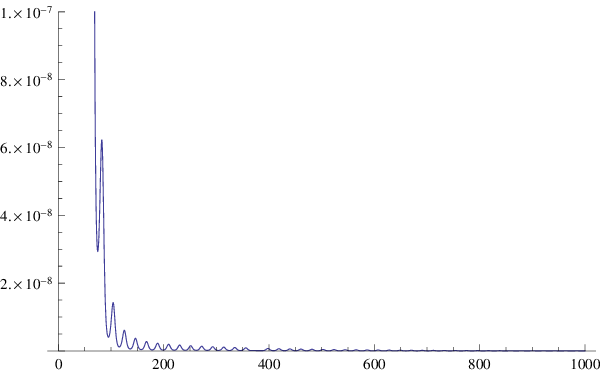}
  \end{minipage}
    \caption{Constant birth and death, quarantine-adjusted - left $\alpha=0.25$ ($R^{S/QA}_p \approx 1.07$) and right $\alpha=0.23$ ($R^{S/QA}_e\approx 0.99$)}
      \label{fig2}
      \end{figure}

Next, we will consider variable birth and death rates, namely we set
    $$\Lambda(t)=.001(1+\sin(.3t))(1-\e^{-t}) \quad \text{and} \quad d(t)=.035(1+\cos(.3t))(1+\e^{-t}).$$
Notice that, for $\lambda_0=2\pi/.3$, we have
$$\lim_{t \to +\infty} \frac{1}{\lambda_0} \int_t^{t+\lambda_0} \Lambda(s) \ds = 0.001 \quad \text{and} \quad \lim_{t \to +\infty} \frac{1}{\lambda_0} \int_t^{t+\lambda_0} d(s) \ds = 0.035$$
and so we can analyse the influence of the nonautonomy of $\Lambda(t)$ and $d(t)$ comparing our next results to the previous ones.

We start with quarantine-adjusted incidence. According to corollary~\ref{teo:nonaut-S-QA}, we have in this case the same thresholds as before. Setting again $\alpha=0.25$ and $\lambda=0.2$, we have the same threshold ($R^{S/QA}_p(\lambda)=1.07527>1$) and the infectives are permanent and setting $\alpha=.23$ and $\lambda=0.2$ we get again the same threshold ($R^{S/QA}_e(\lambda)=0.989247<1$) and we can conclude that the infectives go to extinction. Figure~\ref{fig3} corresponds to this situation.
\begin{figure}[h]
  \begin{minipage}[b][3.5cm]{.49\linewidth}
    \includegraphics[scale=0.58]{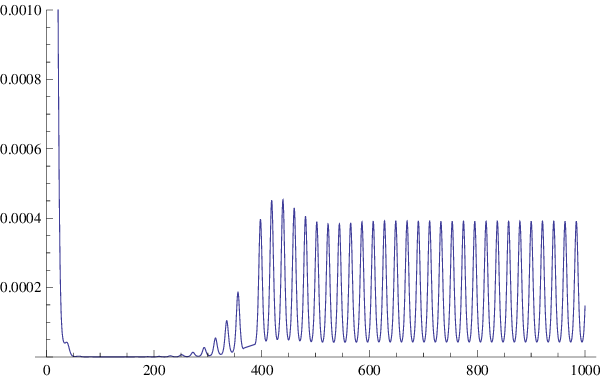}
\end{minipage}
  \begin{minipage}[b][3.5cm]{.49\linewidth}
        \includegraphics[scale=0.58]{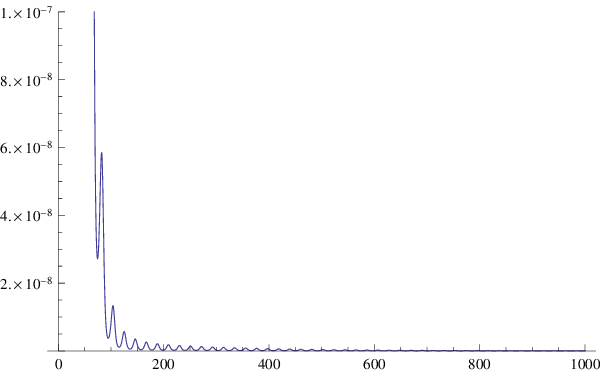}
\end{minipage}
    \caption{Variable birth and death, quarantine-adjusted - left $\alpha=0.25$ ($R^{S/QA}_p \approx 1.07$) and right $\alpha=0.23$ ($R^{S/QA}_e\approx 0.99$)}
          \label{fig3}
\end{figure}

Next we turn our attention to mass-action incidence. Corollary~\ref{teo:nonaut-SIM} doesn't give us the threshold in this case since we had to assume constant birth and death rates in section~\ref{section:GEN}. In this case we need to compute the thresholds numerically. Setting $\lambda=6\pi$ and $\alpha=7.6$ in~\eqref{eq:liminf-threshold} we get $R_p(6\pi)\approx 1.33768 > 1$ and setting $\lambda=6\pi$ and $\alpha=7.3$ in~\eqref{eq:limsup-threshold} we get $R_e(6\pi)\approx 0.900342 <1$. Thus we conclude that for $\alpha=7.6$ we have permanence and for $\alpha=7.3$ we have extinction. Figure~\ref{fig4} supports the conclusions obtained by computing the threshold.

\begin{figure}[h]
  \begin{minipage}[b][3.5cm]{.49\linewidth}
    \includegraphics[scale=0.58]{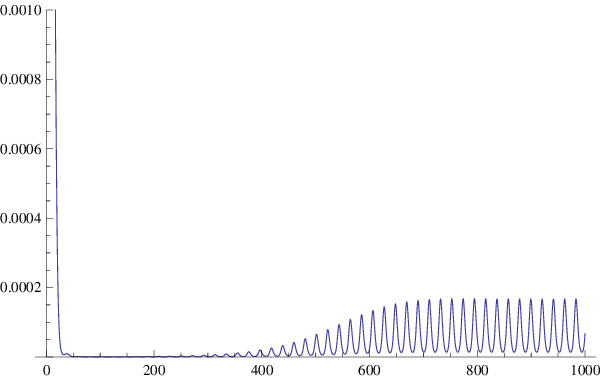}
\end{minipage}
  \begin{minipage}[b][3.5cm]{.49\linewidth}
        \includegraphics[scale=0.58]{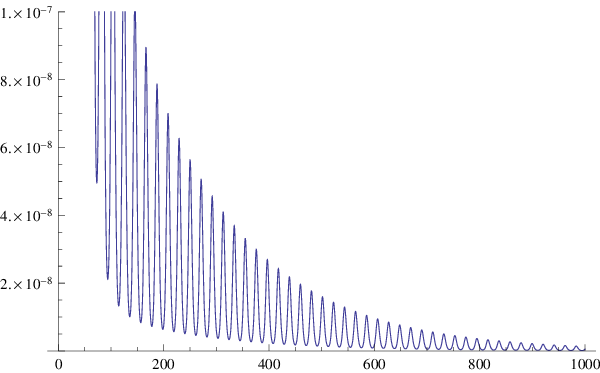}
\end{minipage}
    \caption{Variable birth and death, mass-action - left $\alpha=7.6$ ($R_p(6\pi)\approx 1.34$) and right $\alpha=7.3$ ($R_p(6\pi)\approx  0.90$)}
          \label{fig4}
\end{figure}

\section{Discussion}\label{section:D}

In this work we obtained threshold conditions for the permanence and extinction of the infectives for a nonautonomous model with isolation and general incidence. Our result includes as particular cases several models in the literature and we were able to obtain sharp thresholds based only on some kind of asymptotic averages of the parameters defined in~\eqref{eq:asymptotic-parameters}. Namely, we obtained these kind of thresholds in the autonomous case, the asymptotically autonomous case, the periodic case with constant birth and death rates, the general nonautonomous model with mass-action and constant birth and death rates and the general nonautonomous model with quarantine-adjusted incidence. In all the previous situations we can see that the thresholds obtained correspond to the thresholds for the autonomous model whose parameters are obtained by the corresponding nonautonomous parameter computing one of the numbers in~\eqref{eq:asymptotic-parameters}. For example, we have shown that the thresholds for the asymptotically autonomous model are the same as the thresholds for the corresponding autonomous model and, for the nonautonomous model with quarantine-adjusted incidence $R_p^{S/QA}(\lambda)$ the thresholds are equal to the number $R_p(\lambda)$ corresponding to the autonomous model (with arbitrary $\Lambda,d,\epsilon,\alpha_2>0$)
$$
\begin{cases}
S'=\Lambda-\beta_\lambda^- \frac{SI}{S+I+R}-d_\lambda^- S \\[2mm]
I'=\beta_\lambda^- \frac{SI}{S+I+R}-\left[\gamma_\lambda^-+\sigma_\lambda^-+d_\lambda^-+(\alpha_1)_\lambda^- \right]I \\[2mm]
Q'=\sigma_\lambda^-I -[d_\lambda^-+\alpha_2+\epsilon]Q \\[2mm]
R'=\gamma_\lambda^- I + \epsilon Q -d_\lambda^- R
\end{cases}.
$$
Our simulations showed that the described situation doesn't hold in general. In fact, in section~\ref{section:SIM}, we showed that, unlike the quarantine-adjusted case, for mass-action incidence the thresholds are strongly dependent on the shape of the functions $\beta(t)$, $\lambda(t)$ and $d(t)$ and not only on the numbers $\beta_\lambda^+, \beta_\lambda^-, \Lambda_\lambda^+, \Lambda_\lambda^- d_\lambda^+$ and $d_\lambda^-$.

Besides obtaining thresholds for the permanence and extinction of infectives, Theorem~\ref{teo:Main} also states that, if $R_e(\lambda)<1$, disease-free solutions are globally asymptotically stable. It is natural to question what can be said about the global behavior of solutions when $R_p(\lambda)>1$. This problem seems to be very difficult in this general setting and probably it can depend highly on the incidence rates. In fact, already in the autonomous setting, for $R_0>1$, the mass-action incidence and quarantine-adjusted incidence model are very different: in~\cite{Hethcote-Zhien-Shengbing-MB-2002} it was shown that for mass-action incidence, when $R_0>1$, there is a unique endemic equilibrium (that coexists with the disease-free equilibrium) and is globally asymptotically stable in the region $\{(S,I,Q,R)\in \R^4: S \ge 0 \wedge I>0 \wedge Q\ge 0 \wedge R \ge 0\}$ while for quarantine-adjusted incidence Hopf bifurcations can occur for some parameter values and the endemic equilibrium may became unstable and periodic solutions can occur.

As said in the introduction, to make the model more adapted to the case of diseases with long latent stages, a latent class should be considered. This extra class adds difficulties to the obtention of thresholds and will certainly be the object of future research.

\section{Acknowledgements}\label{section:A}

I would like to thank the anonymous referees for their helpful suggestions which allowed me to improve the paper significantly.

\bibliographystyle{elsart-num-sort}

\end{document}